\documentclass[a4paper,12pt]{elsarticle}
\usepackage{amssymb,amsmath,amsthm}

\makeatletter
\def\ps@pprintTitle{%
  \let\@oddhead\@empty
  \let\@evenhead\@empty
  \def\@oddfoot{\reset@font\hfil\thepage\hfil}
  \let\@evenfoot\@oddfoot
}
\makeatother

\newtheorem{theorem}{Theorem}[section]
\newtheorem{lemma}[theorem]{Lemma}

\newtheorem{problem}[theorem]{Problem}
\newtheorem{conjecture}[theorem]{Conjecture}
\newtheorem{prop}[theorem]{Proposition}
\newtheorem{claim}[theorem]{Claim}
\newtheorem{definition}[theorem]{Definition}
\newtheorem{observation}[theorem]{Observation}

\newcommand\lref[1]{Lemma~\ref{lem:#1}}
\newcommand\tref[1]{Theorem~\ref{thm:#1}}
\newcommand\cref[1]{Corollary~\ref{cor:#1}}
\newcommand\clref[1]{Claim~\ref{clm:#1}}

\newcommand\pref[1]{Proposition~\ref{prop:#1}}

\newcommand\ul[1]{\underline #1}

\newcommand\rarrow{\rightarrow}
\newcommand\ttt{ \in {\cal T}}
\newcommand\qq{ \in {\cal Q}}

\begin{document}
\begin{frontmatter}
\title{Nonrepetitive  colorings of lexicographic product of graphs}
\author[renyi]{Bal\'azs Keszegh \fnref{kesz}}
\author[renyi]{Bal\'azs Patk\'os\fnref{pat}}
\author[znju]{Xuding Zhu\fnref{zhu}}
\address[renyi]{Alfr\'ed R{\'e}nyi Institute of Mathematics, Hungarian Academy of Sciences \\ Re\'altanoda u. 13-15 Budapest, 1053 Hungary\\
        {\tt email}: $<$keszegh.balazs,patkos.balazs$>$@renyi.mta.hu}
\address[znju]{Department of Mathematics, Zhejiang Normal University,
   China.  \\ {\tt email:}
        {\it xudingzhu@gmail.com}}
\fntext[kesz]{Research supported by OTKA, grant NK 78439, the J\'anos Bolyai Research Scholarship of the Hungarian Academy of Sciences, and by OTKA under EUROGIGA project GraDR 10-EuroGIGA-OP-003.}
\fntext[pat]{Research supported by Hungarian NSF, under contract PD-83586, and the J\'anos Bolyai Research Scholarship of the Hungarian Academy of Sciences. Corresponding author.}
\fntext[zhu]{Grant Numbers: NSF11171310 and ZJNSF  Z6110786.}
\begin{abstract}
A coloring $c$ of the vertices of a graph $G$ is  nonrepetitive if
there exists no path $v_1v_2\ldots v_{2l}$ for which
$c(v_i)=c(v_{l+i})$ for all $1\le i\le l$. Given  graphs $G$ and $H$
with $|V(H)|=k$, the lexicographic product $G[H]$   is the graph
obtained by substituting every vertex of $G$ by a copy of $H$,  and
every edge of $G$ by a copy of $K_{k,k}$.
We  prove that for a sufficiently long path $P$, a nonrepetitive
coloring of   $P[K_k]$ needs at least $3k+\lfloor k/2\rfloor$
colors. If $k>2$ then we need exactly $2k+1$ colors to
nonrepetitively color   $P[E_k]$, where $E_k$ is the empty graph on
$k$ vertices. If we further require that every copy of $E_k$ be
rainbow-colored and the path $P$ is sufficiently long, then the
smallest number of colors needed for $P[E_k]$ is at least $3k+1$ and
at most $3k+\lceil k/2\rceil$. Finally, we define fractional
nonrepetitive colorings of graphs and consider the connections
between this notion and the above results.\end{abstract}

\begin{keyword} non-repetitive coloring \sep lexicographic product of graphs\sep fractional relaxation \end{keyword}
\end{frontmatter}

\section{Introduction}
A sequence $x_1\ldots x_{2l}$ is a \textit{repetition} if
$x_i=x_{l+i}$ for all
 $1 \le i \le l$. A sequence is \textit{nonrepetitive} if it does not contain
a string of consecutive entries forming a repetition. In 1906, Thue
\cite{T} found an infinite nonrepetitive sequence using only three
symbols.

Alon, Grytzuk, Ha\l uszczak, Riordan \cite{AGHR} generalized the
notion of nonrepetitiveness to graph coloring: a coloring $c$ of a
graph $G$ is \textit{nonrepetitive } if there is no path $v_1,\ldots
,v_{2l}$ in $G$ such that the string $c(v_1),\ldots ,c(v_{2l})$ is a
repetition. The {\em Thue chromatic number } of   $G$ is the least
integer $\pi(G)$ such that there exists a nonrepetitive coloring $c$
of $G$ using  $ \pi(G) $ colors. With this notation, Thue's result
says $\pi(P_{\infty})=3$ (the fact that 2 colors are not enough can
be easily seen   for a path of length at least 4). A survey and a
good introduction to the topic is \cite{G}.

In this paper we are interested in nonrepetitive coloring of the
lexicographic product of graphs.

\begin{definition}
Let $G=(V_1,E_1)$ and $H=(V_2,E_2)$ be two graphs. The
\textit{lexicographic product} of $G$ and $H$ is the graph $G[H]$
with vertex set $V_1 \times V_2$ and $(v_1,v_2)$ is joined to
$(v_1',v_2')$ if either $(v_1,v_1') \in E_1$ or $v_1=v_1'$ and
$(v_2,v_2') \in E_2$.
\end{definition}

For any vertex $v \in V_1$,  the set $  \{(v,v_2):v_2 \in V_2\}$,
denoted by $v[H]$, is called a \textit{layer} of $G[H]$ and the
subgraph induced by a layer is isomorphic to $H$. If all the
vertices in $v[H]$ are colored by distinct colors, then we say
$v[H]$ is {\em rainbow colored}. A {\em rainbow  nonrepetitive
coloring} of $G[H]$ is a nonrepetitive coloring $c$ of $G[H]$ in
which all the layers are rainbow colored. The {\em rainbow Thue
chromatic number } of $G[H]$ is the least integer $\pi_R(G[H])$ such
that there exists a rainbow nonrepetitive coloring $c$ of $G[H]$
using $ \pi_R(G[H]) $ colors.

Denote by $E_n,K_n,P_n$  the empty graph, the complete graph and the
path on $n$ vertices, respectively.  It follows from the definition
that $\pi(G[E_k]) \le \pi_R(G[E_k]) \le \pi_R(G[K_k])= \pi(G[K_k])$
for any graph $G$ (note that every nonrepetitive coloring of
$G[K_n]$ is a rainbow nonrepetitive colouring).

Non-repetitive coloring of lexicographic product of graphs has not
been studied systematically before. However, a result of Bar\'at and
Wood \cite{baratwood} can be rephrased in our context: in Lemma 2 of
their paper they showed that for any tree $T$ and integer $k$,
$\pi(T[K_k]) \le 4k$. We shall prove that this bound is sharp, by
constructing  a tree $T$ for which $\pi(T[E_k])=4k$ for every
positive integer $k$.

Our main results concentrate on the lexicographic product of   paths
with complete graphs or empty graphs.

\begin{theorem}
\label{thm:2k1} For any $n \ge 4$ and $k\ne 2$,
$\pi(P_n[E_k])=2k+1$. For $k=2$,   $5 \le \pi(P_n[E_2]) \le
6$.
\end{theorem}

\begin{theorem}
\label{thm:rainbow} For any pair of integers $n \ge 24$ and $k\ge
2$, $3k+1 \le \pi_R(P_n[E_k]) \le  3k+\lceil k/2\rceil$.
\end{theorem}

\begin{theorem}
\label{thm:complete} For any   integer $n \ge 28$, $3k+\lfloor k/2\rfloor \le \pi(P_n[K_k]) \le 4k$.
\end{theorem}

\section{Proofs}

We present the proofs of the lower and upper bounds in separate
subsections.  Most lower bounds   rely on the same lemmas. The
proofs for the upper bounds   use earlier ideas and results by
K\"undgen and Pelsmajer \cite{KP}.

\subsection{Lower bounds}
\label{sec:low}

\begin{lemma}
\label{lem:degree} Let $c$ be a nonrepetitive coloring of $G[E_k]$.
If $v \in V(G)$ is a vertex of degree $d$ and  two vertices in
$v[E_k]$ receive the same color, then $c$ uses at least $dk+1$
colors.
\end{lemma}
\begin{proof}
Let $v_1,v_2,\ldots ,v_d$ be the neighbors of $v$ in $G$,  and let
$u_1,u_2 \in v[E_k]$ be vertices with $c(u_1)=c(u_2)$. For any pair
of vertices $w_1,w_2 \in \cup_{i=1}^dv_i[E_k]$, we have $c(w_1)\neq
c(w_2)$, for otherwise the coloring of the path $w_1u_1w_2u_2$ would
be a repetition. Also colors used for vertices in
$\cup_{i=1}^dv_i[E_k]$ are different from that of $u_1$ and $u_2$.
Hence $c$ uses at least $dk+1$ colors.
\end{proof}

\begin{lemma}
\label{lem:path4} Let $P=(v_1v_2v_3v_4)$ be a path of $4$ vertices
in $G$ and $c$ be a nonrepetitive  coloring  of $G[E_k]$. Then
either the color sets of the first three layers are pairwise
disjoint or the color sets of the last three layers are pairwise
disjoint. In particular, if all the four layers are rainbow colored,
then $c$ uses at least $3k$ colors.
\end{lemma}
\begin{proof}
To avoid repetitions of length two,  $c[v_i[E_k]] \cap
c[v_{i+1}[E_k]] =\emptyset$ for all $i=1,2,3$. If $a \in
c[v_1[E_k]]\cap c[v_3[E_k]]$ and $b \in c[v_2[E_k]]\cap
c[v_4[E_k]]$, then there is a path with colors $abab$.  Therefore,
either $c[v_1[E_k]],c[v_2[E_k]], c[v_3[E_k]]$ or
$c[v_2[E_k]],c[v_3[E_k]], c[v_4[E_k]]$ are pairwise disjoint.
\end{proof}

We now construct a tree $T$ with $\pi(T[E_k])$ matches the upper
bound of Bar\'at and Wood \cite{baratwood} mentioned in the
introduction. Let $T_{3,6}$ denote the rooted tree in which  all
non-leaf vertices have degree three, and all leaves have distance
$5$ from root vertex, i.e. $T_{3,6}$ looks like the usual binary
tree except that the root has three children. We will use the
notions children and father in the standard way.

\begin{lemma}
\label{lem:ball}
A rainbow nonrepetitive  coloring $c$ of $T_{3,6}[E_k]$  uses at least $4k$ colors.
\end{lemma}
\begin{proof}
Assume $c$ is a rainbow nonrepetitive coloring of $T_{3,6}[E_k]$
using  at most $4k-1$ colors.

\begin{claim}
Let $v \in V(T_{3,6})$.
\begin{itemize}
\item If $v$ has two children  $v_1,v_2$, with $c[v_1[E_k]]\cap c[v_2[E_k]]\neq \emptyset$,
then for any children $v_3$ of $v_1$ or $v_2$, $c[v[E_k]] \cap c[v_3[E_k]]=\emptyset$.
\item If  $w$ is the father of $v$,   $v_1,v_2$ are children of $v$,
with $c[w[E_k]] \cap (c[v_1[E_k]] \cup c[v_2[E_k]])=\emptyset$, then $c[v_1[E_k]] \cap c[v_2[E_k]]\neq \emptyset$.
\end{itemize}
\end{claim}
\begin{proof}  The first statement is true, for otherwise
there is a path $u_1uu_2u_3$ of size four whose colors form a
repetition, where $u \in v[E_k], u_i \in v_i[E_k]$. The second
statement follows from the pigeon-hole principle and the fact that
 $c[v[E_k]]$ are disjoint from
$c[v_1[E_k]] \cup c[v_2[E_k]] \cup c[w[E_k]]$.
\end{proof}

Let $v_1,v_2,v_3$ be the children of the root $r$. As $c[r[E_k]]
\cap (\cup_{j=1}^3c[v_j[E_k]])=\emptyset$ and the number of colors
used by $c$ is at most $4k-1$, there exist $1 \le i <j \le 3$ with
$c[v_i[E_k]\cap c[v_j[E_k]]\neq \emptyset$. By the Claim, for any
children $w$ of $v_i$ or $v_j$, $c[r[E_k]] \cap
c[w[E_k]]=\emptyset$. Again, as the total number of colors is at
most $4k-1$, for the two children $w_1, w_2$ of $v_i$, we have
$c(w_1) \cap c(w_2) \ne \emptyset$. Repeat this argument, we find a
path $u_0u_1u_2uu_3u_4u_5$ in $T_{3,6}$ such that $u_0$ is the root
of $T_{3,6}$ and $c[u_i[E_k]]$ is disjoint from $c[u_j[E_k]]$ for
$j=i\pm 1,2$. But then again as $c$ uses at most $4k-1$ colors we
find vertices $w_i \in u_i[E_k]$ $i=0,1,\ldots,5$ such that
$c(w_0)=c(w_3),c(w_1)=c(w_4),c(w_2)=c(w_5)$ and thus
$w_0w_1w_2w_3w_4w_5$ is a repetition of size six.
\end{proof}

\begin{lemma}
\label{lem:t47} There exists a tree $T$ such that for any positive
integer $k$, $\pi(T[E_k]) = 4k$.
\end{lemma}
\begin{proof}
Let $T=T_{4,7}$ be the rooted tree in which all non-leaf vertices
have degree four, and all leaves have distance $6$ from the root
vertex. As mentioned above, it was proved by Bar\'at and Wood
\cite{baratwood} that $\pi(T[E_k]) \le 4k$. Let $c$ be a
nonrepetitive coloring of $T[E_k]$. We shall show that at least $4k$
colors are used. If a subgraph of $T_{4,7}[E_k]$ isomorphic to
$T_{3,6}[E_k]$ is rainbow-colored, then we are done by \lref{ball}.
If not, then we are done by \lref{degree}.
\end{proof}

To prove the lower bounds of \tref{rainbow} and \tref{complete} we
need some preparations. Given a nonrepetitive sequence $S$ over $3$
letters $A,B,C$, by a \textit{palindrome} we mean a  subsequence of
$x_1\ldots x_{2l+1}$ of odd length $2l+1\ge 3$ such that $x_i =
x_{2l+2-i}$ for $i=1,2,\ldots, l$. The middle letter $x_{l+1}$ of a
palindrome is called a \textit{peak} of the sequence. In writing a
sequence we emphasize peaks by underlining them. The \textit{gap}
between two consecutive peaks is the number of letters between them
in $S$. For technical reasons, the first and last letter of a
sequence is also regarded as a peak. In other words, a letter is not
a peak if and only if its two neighbors exist and are different.
Two sequences are {\em equivalent} if they are the
same up to a permutation of the letters $A,B$ and $C$.

\begin{lemma}\label{lem:gap}
In a sequence $S$ over $3$ letters that avoids repetitions of length
at most $6$ each gap is at most $3$ and at least $1$, except the
first and the last gap that can be $0$.
\end{lemma}
\begin{proof}
If there is a $0$ gap which is neither the  first gap nor the last
gap,   then there would be a repetition of length $4$ in $S$. To
prove that a gap is at most $3$, observe that between two peaks the
letters are determined by the first peak-letter $x$ and the letter
after $x$. Indeed, without loss of generality, if these letters are
$\ul{A}B$ then as $B$ is not a peak, the third letter is $C$. In
general the next letter is always the letter different from the
previous two letters until we reach the next peak. Thus if there
would be a gap of size $4$ then there would be a sequence equivalent
to $\ul{A}BCABC$ (the last letter may or may not be a peak), which
includes a repetition.
\end{proof}

\begin{lemma}\label{lem:pal}
In a sequence over $3$ letters, if $v$ is a peak with gap $g_1$ on
one side and $g_2 \ge g_1$ on the other side, then  it is the center
of a palindrome of length $2g_1+3$.
\end{lemma}
\begin{proof}
This follows again  from the fact that the peak and its neighbor
 determine all the letters until the next peak (on both sides). So
 going from $v$ to each side,  the $g_1+1$ letters are the same, and
 hence $v$ is the center
of a palindrome of length $2g_1+3$.
\end{proof}

\begin{lemma}\label{lem:peak}
Assume $S$ is a sequence on $3$ letters that avoids repetitions of
length at most $6$. If there are three consecutive gaps $g_1\ge
g_2\le g_3$, then there is a subsequence equivalent to one of the
following
\begin{enumerate}
\item CB\ul{A}B\ul{C}BA
\item ACB\ul{A}BC\ul{A}CBA
\item BACB\ul{A}BCA\ul{B}ACBA.
\end{enumerate}
\end{lemma}
\begin{proof}
By \lref{gap}, $g_2=1,2$ or $3$. By observing that letters between
two peaks are determined by the  peak-letters and the letter besides
the peak letters, it is easy to verify that if $g_2=1$
(respectively, $g_2=2$ or $g_2=3$), then the resulting subsequence
is as the first (respectively, the 2nd or the 3rd) listed above.
Note that the first and last letters in these sequences might be
also peaks.
\end{proof}

\begin{lemma}\label{lem:find}
Given a sequence $S$ of length $22$ on $3$ letters that avoids
repetitions of length at most $6$, there exist three consecutive
gaps $g_1\ge g_2\le g_3$.
\end{lemma}
\begin{proof}
By \lref{gap}, the series of gaps contains only the numbers
$0,1,2,3$. Suppose that the sequence $S$ does not contain  three
consecutive gaps $g_1\ge g_2\le g_3$.  Then $0$ can only be the
length of the first or the last gap, a gap of length $1$ must be
adjacent to a gap of length $0$, a gap of length $2$ must be
adjacent to a gap of length at most $1$, and a gap of length $3$
must be adjacent to a gap of length at most $2$.  The longest such
sequence of gaps is the following: $0,1,2,3,3,2,1,0$. Thus the
sequence can have length at most $12+9=21$ (the number of letters in
gaps plus the number of peak letters interceding them).
\end{proof}

\begin{lemma}
\label{lem:rlower}
For $k \ge 2$, $ \pi_R(P_{24}[E_k]) \ge 3k+1$.
\end{lemma}
\begin{proof}
Let  $P_{24}=p_1p_2\dots p_{24}$ and $G=P_{24}[E_k]$. For simplicity
we denote the layer corresponding to $p_i$ by $V_i$. Suppose $G$ has
a nonrepetitive rainbow $3k$-coloring. By \lref{path4}, all $3k$
colors are used. We distinguish two cases.

\vskip 0.2truecm

\textsc{Case A}: There exists an index  $2\le j\le 21$ such that
$c[V_j]\ne c[V_{j+2}]$ and $c[V_{j}] \cap c[V_{j+2}]\neq \emptyset$.

\vskip 0.2truecm

Suppose first that $2 \le j\le 19$. Let $b$ be a color in $c[V_j]
\cap c[V_{j+2}]$. By \lref{path4}, $c[V_{j-1}] \cap
c[V_{j+1}]=\emptyset$ and $c[V_{j+1}] \cap c[V_{j+3}]=\emptyset$.

As both $\{c[V_{j-1}],c[V_{j}],c[V_{j+1}]\}$ and
$\{c[V_{j+1}],c[V_{j+2}],c[V_{j+3}]\}$ partition the colors into $3$
parts of size $k$ and $c[V_{j}]\ne c[V_{j+2}]$, there exist colors
$d\in c[V_{j}]\cap c[V_{j+3}]$, $e\in c[V_{j-1}]\cap c[V_{j+2}]$ and
$f\in c[V_{j-1}]\cap c[V_{j+3}]$. Now $c[V_{j+4}]$ must be disjoint
from $c[V_{j+1}]$, as a color $a$ appearing in both $c[V_{j+1}]$ and
$c[V_{j+4}]$ would yield a repetition $edaeda$ of colors on
$c[V_{j-1}]$, $c[V_{j}]$, $c[V_{j+1}]$, $c[V_{j+2}]$, $c[V_{j+3}]$,
$c[V_{j+4}]$. As $c[V_{j+4}]$ is also disjoint from $c[V_{j+3}]$ we
must have $c[V_{j+4}]=c[V_{j+2}]$. As $k \ge 2$, there are colors
$b, h\in c[V_{j+2}] = c[V_{j+4}]$.

Now, $c[V_{j+5}]$ is disjoint from $c[V_{j+4}]$ and also disjoint
from $c[V_{j+3}]$ (as otherwise there would be a repetition $hdhd$).
Thus $c[V_{j+5}]=c[V_{j+1}]$. Picking a color $a \in
c[V_{j+5}]=c[V_{j+1}]$ we obtain a repetitively colored path
$v_{-1},v_{0},v_1,v_2,v_{3},v_{4},v_{5},v'_{4}$ ($v_i\in V_{j+i}$
and $v'_4\in V_{j+4}$) with colors $fbahfbah$, a contradiction.

This proof works only if $2 \le j\le 19$, as we used the existence
of $V_{j-1},\dots ,V_{j+5}$. Yet a symmetric reasoning works in case
$6 \le j\le 23$, thus covering the whole range of possible values of
$j$.

\vskip 0.2truecm

\textsc{Case B}: For each $2\le j\le 21$, either
$c[V_{j}]=c[V_{j+2}]$ or $c[V_{j}]\cap c[V_{j+2}]=\emptyset$.

\vskip 0.2truecm

First we prove that there exists a partition $A \cup B \cup C$ of
the $3k$ colors such that for every $2\le j\le 21$,  $c[V_j]=A$ or $
B$ or $ C$. Indeed, write $A=c[V_2]$ and $B=c[V_3]$. We prove by
induction that for every $4\le j\le 21$, $c[V_j]$ equals to one of
$A,B,C$. To avoid repetitions of size two $c[V_{j}]$ must be
disjoint from $c[V_{j-1}]$ and if it is not the same as
$c[V_{j-2}]$, then by the assumption of \textsc{Case B},  $c[V_j]
\cap c[V_{j-2}]=\emptyset$. As there are only $3k$ available colors,
$c[V_{j}]$ must be equal to the third color set (the one different
from $c[V_{j-1}]$ and $c[V_{j-2}]$, which by induction are two color
sets from $A,B,C$).

Thus the coloring of the layers from $j=2$ to $j=23$ can be regarded
as a sequence on the three letters $A,B,C$, which has length $22$.
Observe that this sequence is repetition-free,  as otherwise there
would be a repetitive path in the coloring of the original graph. By
\lref{peak} and \lref{find}, there is a subsequence of the form
$CBABCBA$ or $ACBABCACBA$ or $BACBABCABACBA$.

Each of $A,B,C$ contains $k \ge 2$ colors. Let $a_1, a_2$
(respectively, $b_1, b_2$ and $c_1, c_2$) be two distinct colors in
$A$ (respectively, $B$ and $C$). Then a path of color sequence
$b_1c_1b_2a_1b_1c_1b_2a_1$ can be found from the parts with color
sequence  $CBABCBA$. Indeed, to find this start from the second part
(which has color set $B$), go to the first part (which has color set
$A$), then follow the original path to the end. Similarly, paths of
color sequences $b_1c_1a_1c_2b_2a_1b_1c_1a_1c_2b_2a_1$ and
$b_1c_1a_1b_1a_2c_2b_2a_1b_1c_1a_1b_1a_2c_2b_2a_1$ can be found from
the parts with color sequence $ACBABCACBA$ and $BACBABCABACBA$,
respectively.
\end{proof}

\begin{theorem}
\label{thm:clower}
For any integer $k \ge 1$, $\pi(P_{28}[K_k])\ge 3k+\lfloor k/2\rfloor$.
\end{theorem}
\begin{proof}
Assume to the contrary that there is a nonrepetitive coloring $c$ of
$G=P_{28}[K_k]$ with  $3k+\lfloor k/2\rfloor -1$ colors. The
vertices of $P_{28}$ are $v_1, v_2, \ldots, v_{28}$. Let $X_i =
c(v_i[K_k])$. So each $X_i$ is a $k$-subset of the $3k+\lfloor
k/2\rfloor -1$ colors. For the remainder of this proof,    a
\textit{set of colors} means a $k$-subset of the set of the
$3k+\lfloor k/2\rfloor -1$ colors. For two sets of colors $X$ and
$Y$, we say $X$ is \textit{$Y$-rich} (and $Y$ is $X$-rich) if
$|X\cap Y|\ge \lceil k/2+1\rceil$. We write $XYZ \ttt$ if $X,Y,Z$
are three pairwise disjoint color sets, and write $XYZW \qq$ if $XYZ
\ttt$ and $YZW \ttt$. We shall frequently use the following
observation.
\begin{prop}
\label{prop:useful} If $Y$ is $X$-rich and $Z$ is $Y$-rich then
$|X\cap Z|\ge 2$. If $XYZW \qq$ then $W$ is $X$-rich.
\end{prop}

\begin{claim}
\label{clm-new} Assume $P_9[K_k]$ is nonrepetitively colored with
$3k+\lfloor k/2\rfloor -1$ colors, and the color sets of the layers
are $XYABCDEFG$ and  $ABC \ttt$.
\begin{enumerate}
\item[(1)] If  $DEF \ttt$, then
$D, E, F$ are either $B,A,C$-rich respectively, or $A,C,B$-rich,
respectively.
\item[(2)] If $D \cap F \ne \emptyset$, then $EFG \ttt$ and one of the following holds:
\begin{itemize}
\item[(i)]  $F$ is $D$-rich and
$D, E, F, G$ are $A,B,A,C$-rich, respectively.
\item[(ii)]
$G$ is $D$-rich and $D, E, F, G$ are $B,A,C,B$-rich, respectively.
\end{itemize}
\end{enumerate}
\end{claim}

The proof of this claim is postponed to the next subsection. Now we use this claim and
continue with the proof of \tref{clower}.

We (partially) label the sequence $X_3X_4\ldots X_{28}$ by three
labels as follows: The first three consecutive pairwise disjoint
color sets are labeled $A,B,C$, respectively. In  other words, if
$X_3X_4X_5 \ttt$, then $X_3, X_4, X_5$ are labeled $A,B,C$,
respectively. Otherwise,  $X_4X_5X_6 \ttt$, then $X_4, X_5, X_6$ are
labeled $A,B,C$, respectively, and $X_3$ is unlabeled. Suppose we
have already labeled  $X_3X_4\ldots X_{i}$ (with $X_3$ possibly
unlabeled).   Let $j$ be the largest index such that $j \le i$ and
$X_{i+1}$ is $X_j$-rich. We label $X_{i+1}$ the same label as $X_j$.
By Claim \ref{clm-new}, we can label three or four consecutive color
sets simultaneously at each step. Note that by using Claim
\ref{clm-new} to label three or four consecutive color sets, the
last three consecutive color sets are always pairwise disjoint. So
we can repeatedly apply
 Claim \ref{clm-new} to label the next three or
four consecutive color sets. Thus the labeling
is well-defined,  except possibly the
 last three color sets are unlabeled.

Denote by $S$ the label sequence constructed above, which has length
at least $22$ (the first color set and the last three color sets may
not be labeled). The following observation follows from the
definition.

\begin{observation}
\label{obs-new}
If two color sets $X_i$ and $X_j$ have the same label and there is at most one other
color set between them that gets the same label, then $|X_i \cap X_j| \ge 2$.
\end{observation}

In particular,  if $|i-j| \le 3$ and $X_i$ and $X_j$ have the same label,  then $|X_i \cap X_j| \ge 2$.
Therefore, if $S$ has a repetition of length at most $6$, then it
 yields a repetitive path in $G$ of length at most $6$ along the corresponding layers.
 Thus $S$ contains no repetition of length at most $6$.
By \lref{peak} and \lref{find}, there exists a subsequence $S'$ that
is equivalent to one of the following sequences:

\vskip 0.2truecm

\textsc{Case (i)} $S'=CB\ul{A}B\ul{C}BA$

We write the sequence of color sets corresponding to $S'$ as
$CBAB_1C_1B_2A_1$. By   Observation \ref{obs-new}, there is a
repetitive path in $G$ with colors  $cbab'cbab'$ where $c\in
C,C_1;b\in B,B_2;a\in A,A_1,b'\in B_1,B_2$.

\vskip 0.2truecm

\textsc{Case (ii)} $S'=ACB\ul{A}BC\ul{A}CBA$

We write the sequence of color classes of the layers corresponding
to $S'$ as $ACBA_1B_1C_1A_2C_2B_2A_3$. Again it follows from
Observation \ref{obs-new} that there is a repetitive path in  $G$
with colors $acba'b'c'acba'b'c'$ where $a\in A,A_2;c\in C,C_2;b\in
B,B_2,a'\in A_1,A_3,b'\in B_1,B_2,c'\in C_1,C_2$.

\vskip 0.2truecm

\textsc{Case (iii)} $S'=BACB\ul{A}BCA\ul{B}ACBA$

We write the sequence of color sets corresponding to $S'$ as
$B_0A_0C_0B_1A_1B_2C_1A_2B_3A_3C_2B_4A_4$.

We claim that there is a repetitive path in $G$ with colors
$bacb'a'b''c'a''bacb'a'b''c'a''$ where $b\in B_0,B_3;a\in
A_0,A_3;c\in C_0,C_2;b'\in B_1,B_4,a'\in A_1,A_4,b''\in
B_2,B_4,c'\in C_1,C_2;a''\in A_2,A_3$. For this purpose, it suffices
to show that in each pair of layers from which we need to pick
vertices with the same color, we have at least two possible choices.
This follows from  \pref{useful} if the two layers correspond to
$Y_i$ and $Y_{i+1}$ or $Y_i$ and $Y_{i+2}$ for some letter $Y \in
\{A,B,C\}$. There are some pairs of the form $Y_i$ and $Y_{i+3}$
with $Y \in \{A, B\}$ for which we need to  pick vertices with the
same color. Hence we need to show that $|Y_i \cap Y_{i+3}| \ge 2$
for these pairs. For this purpose, by  \pref{useful},
 it suffices to show that
either $Y_i$ is $Y_{i+2}$-rich or $Y_{i+1}$ is $Y_{i+3}$-rich.  The
required properties follow from the following claim.

\begin{claim}\label{clm:rich}
$B_1$ is $B_3$-rich and $A_1$ is $A_3$-rich.
\end{claim}

\begin{proof}[Proof of Claim]
Consider the reverse of the subsequence
$C_0B_1A_1B_2C_1A_2B_3A_3$. Since $A_2$ is $A_3$-rich, by \lref{path4}, $C_1,A_2,B_3$ are pairwise disjoint.
Apply Claim \ref{clm-new} to the reverse of $C_0B_1A_1B_2C_1A_2B_3$, we conclude that $B_1$ is $B_3$-rich.
Similarly, by \lref{path4}, $A_1,B_2,C_1$ are pairwise disjoint, and apply Claim \ref{clm-new} to
$A_1B_2C_1A_2B_3A_3C_2$, we know that $A_1$ is $A_3$-rich.
\end{proof}

This completes the proof of \tref{clower} (except that the proof of
Claim \ref{clm-new} will be given in the next subsection).
\end{proof}

\subsection{Proof of Claim \ref{clm-new}}

Claim \ref{clm-new} follows from the following three lemmas.

\begin{lemma}
\label{lem:def}
Assume $P_6[K_k]$ is nonrepetitively colored with $3k+\lfloor k/2\rfloor -1$ colors, and the color
sets of the layers are $ABCDEF$. If $ABC \ttt$ and $DEF \ttt$, then
$D, E, F$ are either $B,A,C$-rich respectively, or $A,C,B$-rich respectively.
\end{lemma}

\begin{proof}
We consider three cases.

\vskip 0.2truecm

\textsc{Case 1}: $D\cap A=\emptyset$.
\vskip 0.2truecm

$BACD \qq$ implies that $D$ is $B$-rich.
As $D \cap B \ne \emptyset$, by \lref{path4},  $E \cap C = \emptyset$.
Now $ACDE \qq$,  implies that $E$ is $A$-rich, and
$CDEF \qq$ implies that  $F$ is $C$-rich.

\vskip 0.2truecm

\textsc{Case 2}: $D\cap B=\emptyset$.
\vskip 0.2truecm

$ABCD \qq$ implies that $D$ is $A$-rich. If $E$ intersects both $B$
and $C$, then there is a repetitive path  $abcabc$ where $a\in A,
D$, $b\in B,E$ and $c\in C,E$, a contradiction. If $E$ is disjoint
from $B$, then $CBDE \qq$ implies that  $E$ is $C$-rich, and $BDEF
\qq$ implies that $F$ is $B$-rich. So $D,E,F$ are $A,C,B$-rich,
respectively, and we are done. If $E$ is disjoint from $C$, then
$BCDE \qq$ implies that $E$ is $B$-rich, and $CDEF \qq$ implies that
$F$ is $C$-rich. But then there is a repetition $abcabc$, $a \in
A,D; b\in B,E;c\in C,F$.

\vskip 0.2truecm

\textsc{Case 3}: $D\cap A\ne \emptyset$ and $D \cap B \ne \emptyset$.
\vskip 0.2truecm

In this case, $E \cap C = \emptyset$, for otherwise there is a
repetition $bcbc$, $b\in B,D;c\in C,E$. Now $CDEF \qq$ implies that
$F$ is $C$-rich. This implies that $E\cap B = \emptyset$, for
otherwise there would be a repetition $abcabc$, $a \in A,D; b\in
B,E;c\in C,F$. Then  $ABCE \qq$ implies that  $E$ is $A$-rich, and
$BCED \qq$ implies that  $D$ is $B$-rich.
\end{proof}

\begin{lemma}
\label{lem:defg1}
Assume $P_7[K_k]$ is nonrepetitively colored with $3k+\lfloor k/2\rfloor -1$ colors, and the color
sets of the layers are $ABCDEFG$. If $ABC \ttt$ and $D \cap F \ne \emptyset$, then
$D, E, F, G$ are either $F,B,A,C$-rich respectively, or $B,A,C,B$-rich respectively.
\end{lemma}
\begin{proof}
By \lref{path4}, we know that  $EFG \ttt$ and $CDE \ttt$.
We consider three cases

\vskip 0.2truecm

\textsc{Case 1}: $D\cap B=\emptyset$.
\vskip 0.2truecm

As $BCD \ttt$, we can apply Lemma \ref{lem:def} to the color set sequence $BCDEFG$. Thus $E,F,G$ are either
 $C,B,D$-rich respectively, or $B,D,C$-rich respectively.
Also $ABCD \qq$ implies that $D$ is $A$-rich, and $BCDE \qq$ implies
that $E$ is $B$-rich. Therefore $E$ cannot be $C$-rich, as $B \cap C
= \emptyset$. So $D,E,F,G$ are   $A,B,D,C$-rich respectively. This
implies that $F \cap B =\emptyset$, for otherwise there is a
repetitive path with colors $abb'cabb'c$, $a\in A,D;b\in B,E;b'\in
B,F; c\in C,G$. Also $F \cap C =\emptyset$, for otherwise there is a
repetitive path with colors $abcabc$, $a\in A,D;b\in B,E;c\in C,F$.
Now $ABCF \qq$ implies that $F$ is $A$-rich. Thus we have proved
that $D,E,F,G$ are $A,B,A,C$-rich, respectively, and $F$ is
$D$-rich.

\vskip 0.2truecm

\textsc{Case 2}: $D\cap A=\emptyset$.
\vskip 0.2truecm

Then $BACD \qq$ implies that $D$ is $B$-rich.
As $CDE \ttt$, $E \cap C = \emptyset$.
Thus $ACDE \qq$ and hence  $E$ is $A$-rich.

If $F \cap B \ne \emptyset$, then $(F \cup G) \cap C = \emptyset$, for otherwise
there is a repetitive path with colors $bab'cbab'c$, $b\in B, D,; a\in A,E;b'\in B,F; c\in C,F\cup G$.
Then $(E \cup F \cup G) \cap C = \emptyset$, which is a contradiction as $EFG \ttt$.
So $F \cap B = \emptyset$.

\vskip 0.2truecm

\textsc{Case 2(i)}: $E \cap B \ne \emptyset$. \vskip 0.2truecm Then
$F \cap C = \emptyset$, for otherwise  there is a repetitive path
with colors $bab'cbab'c$, $b\in B,D;a\in A,E;b'\in B,E; c\in C,F$.
Now $ABCF \qq$ implies that $F$ is $A$-rich, which is a
contradiction as $E$ is $A$-rich and $E\cap F = \emptyset$.

\vskip 0.2truecm

\textsc{Case 2(ii)}: $E \cap B = \emptyset$.
\vskip 0.2truecm

Now $BEFG \qq$ implies that $G$ is $B$-rich, and $CBEF \qq$ implies that $F$ is $C$-rich.
So we have proved that $D,E,F,G$ are $B,A,C,B$-rich, respectively.

\vskip 0.2truecm
\textsc{Case 3}: $D\cap A\ne \emptyset$ and $D \cap B \ne \emptyset$.
\vskip 0.2truecm

\textsc{Case 3(i)}: $E \cap A = \emptyset$. \vskip 0.2truecm Now
$BACE \qq$ implies that $E$ is $B$-rich, and $ACED\qq$ implies  that
$D$ is $A$-rich. This implies that $F$ is disjoint from $C$,  for
otherwise there is a repetition $abcabc$, $a \in A,D; b\in B,E;c\in
C,F$. Then $ACEF \qq$ implies that $F$ is $A$-rich, and $DECF \qq$
implies that $F$ is $D$-rich, and $CEFG \qq$ implies that $G$ is
$C$-rich.Thus $D,E,F,G$ are $A,B,A,C$-rich respectively, and we are
done.

\textsc{Case 3(ii)}: $E\cap B=\emptyset$.

Then $ABCE \qq$ implies that $E$ is $A$-rich, and $BCED\qq$ implies
that $D$ is $B$-rich. If $F \cap B = \emptyset$, then $CBEF \qq$
implies that $F$ is $C$-rich and $BEFG\qq$ implies that $G$ is
$B$-rich. So $D,E,F,G$ are $B,A,C,B$-rich, respectively. Thus we
assume $F \cap B \ne \emptyset$. Then $(F \cup G) \cap C =
\emptyset$, for otherwise there is a repetitive path with colors
$bab'cbab'c$, $b\in B,D;a\in A,E;b'\in B,F; c\in C,F \cup G$. Now
$(D\cup E \cup F \cup G) \cap C = \emptyset$, which is a
contradiction.

\textsc{Case 3(iii)}: $E\cap A\ne\emptyset$ and $E\cap B\ne \emptyset$.

In this case, $F \cap C = \emptyset$, for otherwise there is a
repetitive path with colors $abcabc$, $a\in A,D;b\in B,E;c\in C,F$.
If $F\cap B\ne \emptyset$ then $G\cap C =\emptyset$, for otherwise
there is a repetitive path with colors $abb'cabb'c$, $a\in A,D;b\in
B,E;b'\in B,F; c\in C,G$. Then $(D \cup E \cup F \cup G) \cap
C=\emptyset$, which is a contradiction. Thus $F\cap B= \emptyset$.
Now $ABCF\qq$ implies  that $F$ is $A$-rich, and
 $BCFE \qq$ implies that $E$ is $B$-rich, and $CEFG \qq$ implies  that $G$ is $C$-rich, and $DCEF \qq$ implies that $F$ is $D$-rich.
 So $D,E,F,G$ are $F,B,A,C$-rich, respectively.
\end{proof}

\begin{lemma}
\label{lem:defg2}
Assume $P_9[K_k]$ is nonrepetitively colored with $3k+\lfloor k/2\rfloor -1$ colors, and the color
sets of the layers are $XYABCDEFG$, and $ABC \ttt$ and $D \cap F \ne \emptyset$.
If $D, E, F, G$ are $F,B,A,C$-rich, respectively, then $D$ is $A$-rich.
If $D, E, F, G$ are   $B,A,C,B$-rich, respectively, then $G$ is $D$-rich.
\end{lemma}
\begin{proof}
Observe that we started the labeling process without using $X_1,X_2$
for the purpose that we can always find the color sets $X,Y$ used in
this lemma. Assume  $D, E, F, G$ are $F,B,A,C$-rich, respectively.
Apply Lemma \ref{lem:def} to the color set sequence $EDCBAY$ (if
$BAY \ttt$)  or  Lemma \ref{lem:defg1} to the color set sequence
$EDCBAYX$ (if $B \cap Y \ne \emptyset$), the only case not leading
to contradiction gives that  $A$ is  $D$-rich.

Assume $D, E, F, G$ are   $B,A,C,B$-rich, respectively. Apply  Lemma
\ref{lem:def} to the color set sequence $GFEDCB$  (if $DCB \ttt$)
or Lemma \ref{lem:defg1} to the color set sequence $GFEDCBA$ (if $D
\cap B \ne \emptyset$), we conclude that $D$ is either $G$-rich or
$F$-rich. If $D$ is $F$-rich then as $F$ is $C$-rich, $D$ would
intersect $C$, a contradiction.  Thus $D$ is $G$-rich which
completes the proof of the lemma.
\end{proof}

\vskip 0.2truecm
\subsection{Upper bounds}
\label{sec:upp}

 Before we start our proofs, we describe some tools from the  paper of K\"undgen and Pelsmajer \cite{KP}.

\begin{lemma} [K\"undgen, Pelsmajer, Lemma 3 in \cite{KP}]
\label{lem:walk} If $c$ is a nonrepetitive palindrome-free coloring of a path $P$, and $P'$
is obtained from $P$ by adding a loop at
each vertex, then every repetitively colored walk $W_1W_2$ in $P'$
satisfies $W_1 = W_2$.
\end{lemma}

Let $V_1, \ldots ,V_m$ be a partition of $V(G)$ and let $G_k$ and
$G_{>k}$ denote the subgraphs of $G$ induced by $V_k$ and $V_{k+1}
\cup \ldots \cup V_m$, respectively. The \textit{$k$-shadow} of a
subgraph $H$ of $G$ is the set of vertices in $V_k$ which have a
neighbor in $V(H)$. We say that $G$ is \textit{shadow complete}
(with respect to the partition) if the $k$-shadow of every component
of $G_{>k}$ induces a complete graph.

\begin{theorem} [K\"undgen, Pelsmajer, Theorem 6 in \cite{KP}]
\label{thm:4times} If $G$ is shadow complete and each $G_k$ has a
nonrepetitive coloring with $b$ colors, then $G$ has a nonrepetitive
coloring with $4b$ colors.
\end{theorem}

\begin{proof}[Proof of \tref{2k1}] Recall that we want to prove that
for any $n \ge 4$ and $k\ne 2$, we have $\pi(P_n[E_k])=2k+1$ and for $k=2$ we
have $5 \le \pi(P_n[E_2]) \le 6$. The lower bounds of the theorem follow from \lref{degree} and \lref{path4}.

To prove the upper bounds we need to define a nonrepetitive coloring $c$ of
 $P_{\infty}[E_k]$. For $k\ge 3$ let $Y$ denote the set $\{k+1,k+2,\ldots ,2k+1\}$
 and $X$ denote the set $\{1,2,\ldots ,k\}$. If $k=2$, then let $X=\{1,2\},Y=\{3,4,5,6\}$.
 Elements of $Y$ will be denoted by lower case letters $a,b,c,a_1,$ etc. Let $S=s_1s_2s_3s_4\ldots .$
 be an infinite palindrome-free nonrepetitive sequence. Such a sequence exists using only
 4 symbols \cite{AGHR}. Thus we can pick all $s_i$'s from $Y$. Let the vertex set of $P_{\infty}$
 be $\{v_1,v_2,\ldots \}$ and $E(P_n)=\{(v_i,v_{i+1}):1\le i \}$. If $j=4(i-1)+1$, then define $c$
 on $v_j[E_k]$ such that $c[v_j[E_k]]=X$. If $j=4(i-1)+2$ or $j=4i$, then for any vertex $u \in v_j[E_k]$ let
$c(u)=s_i$. Finally, if $j=4(i-1)+3$, then define $c$ on $v_j[E_k]$ such that
 $c[v_j[E_k]]$ is a $k$-subset of $Y\setminus s_i$ (note that if $k \ge 3$,
 then $|Y\setminus s_i|= k$ and if $k=2$, then  $|Y\setminus s_i|= 3$).

We claim that $c$ is nonrepetitive. Assume to the contrary   that
there is a path $Q_1Q_2$ in $P_{\infty}[E_k]$ such that the sequence
of colors on $Q_1Q_2$ is a repetition. Remove all vertices from
$Q_1Q_2$ that have colors from the set $X$. The sequence of colors
of the remaining vertices $Q'_1Q'_2=(q'_{1,1}\ldots
q'_{1,l}q'_{2,1}\ldots q'_{2,l})$ still form a repetition. Let
$P'_{\infty}$ be an infinite path with one loop added to each of its
vertices. Furthermore, let $c_S$ be the coloring of $P'_\infty$ with
$c_S(p'_j)=s_j$. Let us define the function $f:Q'_1Q'_2 \rightarrow
P'_\infty$ with $f(q)=p'_i$ if and only if $q \in v_{4(i-1)+2}[E_k]
\cup v_{4(i-1)+3}[E_k] \cup v_{4i}[E_k]$. Writing $W_1$ and $W_2$
 for the images of $Q'_1$ and $Q'_2$, we obtain that $W_1W_2$ is a walk in $P'_{\infty}$.
\begin{claim}
\label{clm:good1} The sequence of colors of vertices in $W_1W_2$ with respect to the
coloring $c_S$ is a repetition.
\end{claim}
\begin{proof}
Let $1 \le m \le l$. Consider the largest parts of $Q_1$ and $Q_2$
that contain $q'_{1,m}$ and $q'_{2,m}$ such that they form a subpath
of $Q'_1$ and $Q'_2$, i.e. the subpaths of $Q_1$ and $Q_2$ that lie
between consecutive $X$-colored vertices of $Q_1$ and $Q_2$.
Clearly, the part in $Q_1$ lies entirely within $v_{4(i-1)+2}[E_k]
\cup v_{4(i-1)+3}[E_k] \cup v_{4i}[E_k]$ for some $i$ and the part
in $Q_2$ lies entirely within $v_{4(j-1)+2}[E_k] \cup
v_{4(j-1)+3}[E_k] \cup v_{4j}[E_k]$ for some $j$ and vertices of the
former are mapped by $f$ to $p'_i$ and those of the latter are
mapped by $f$ to $p'_j$. If these paths are $(q'_{1,m_1}\ldots
q'_{1,m}\ldots q'_{1,m_2})$ and $(q'_{2,m_1}\ldots q'_{2,m}\ldots
q'_{2,m_2})$, then $c(q'_{1,m_1})=c(q'_{2,m_1})$ and
$c(q'_{1,m_2})=c(q'_{2,m_2})$ and at least one of the pairs
$(q'_{1,m_1},q'_{2,m_1}), (q'_{1,m_2},q'_{2,m_2})$, say the former
one, lie next to an $X$-colored vertex and therefore their $c$-color
is $s_i$ and $s_j$. This shows that
$c_S(f(q'_{1,m}))=s_i=c(q'_{1,m_1})=c(q'_{2,m_1})=s_j=c_S(f(q'_{2,m}))$.
\end{proof}
By \clref{good1} and \lref{walk}, $W_1=W_2$. Suppose first that
$W_1=W_2$   contains at least two different vertices. This means
that the original paths $Q_1$ and $Q_2$ had to cross from
$v_{4(i-1)+2}[E_k] \cup v_{4(i-1)+3}[E_k] \cup v_{4i}[E_k]$ to
$v_{4(i)+2}[E_k] \cup v_{4(i)+3}[E_k] \cup v_{4(i+1)}[E_k]$ or vice
versa.
 But as the layer $v_{4i+1}[E_k]$ is rainbow colored with colors in $X$,
 the original color sequence of $Q_1Q_2$ could not be a repetition.

Suppose then that $W_1W_2$ is a walk repeating the same vertex $p'_i$.
Then all vertices of $Q_1Q_2$ must lie in $v_{4(i-1)+1}[E_k]\cup v_{4(i-1)+2}[E_k] \cup v_{4(i-1)+3}[E_k] \cup v_{4i}[E_k] \cup v_{4i+1}[E_k]$.
Therefore $Q_1Q_2$ cannot contain any vertex from $v_{4(i-1)+3}[E_k]$ as they have unique colors among
vertices in these 5 layers, preventing the possibility of a repetition. By connectivity, we get that
$Q_1Q_2$ must lie either in $v_{4(i-1)+1}[E_k]\cup v_{4(i-1)+2}[E_k]$ or in $v_{4i}[E_k] \cup v_{4i+1}[E_k]$,
say the former. By connectivity, $Q_1Q_2$ must contain a vertex from $v_{4(i-1)+1}[E_k]$ which has a unique
color among vertices in $v_{4(i-1)+1}[E_k]\cup v_{4(i-1)+2}[E_k]$. This contradicts the fact that the color
sequence of $Q_1Q_2$ is a repetition. This finishes the proof of \tref{2k1}.
\end{proof}

\begin{proof}[Proof of \tref{rainbow}]
We will construct a nonrepetitive rainbow coloring $c$ of
$P_{\infty}[E_k]$ with $\lceil 7k/2\rceil$ colors. Let us denote the
vertices of $P_{\infty}[E_k]$ by $p_i$ $i=1,2,3,\ldots $ with
$(p_i,p_j)$ forming an edge if and only if $|i-j|=1$. We will write
$V_i=p_i[E_k]$. Let $X,A,B,C,D,E$ be pairwise disjoint sets with
$|X|=k$, $|B|=|C|=|D|=\lceil k/2\rceil$, $|A|=|E|=\lfloor
k/2\rfloor$. Let $S=s_1s_2s_3\ldots $ be an infinite palindrome-free
nonrepetitive sequence with $s_i \in \{1,2,3,4\}$ for all positive
integers $i$. We define a coloring of $P_{\infty}[E_k]$ using colors
$X \cup A \cup B \cup C \cup D \cup E$ as follows:
\begin{itemize}
\item If $j=4(i-1)+1$ then $c[V_j]=X$.
\item If $s_i=1$, then $c[V_{4(i-1)+2}]=c[V_{4i}]=A \cup B$ and $c[V_{4(i-1)+3}]$ is a $k$-subset of $C \cup D$.
\item If $s_i=2$, then
$c[V_{4(i-1)+2}]=c[V_{4i}]=A \cup C$ and $c[V_{4(i-1)+3}]=B \cup E$.
\item If $s_i=3$, then  $c[V_{4(i-1)+2}]=c[V_{4i}]=C \cup E$ and $c[V_{4(i-1)+3}]=A \cup D$.
\item If $s_i=4$, then  $c[V_{4(i-1)+2}]=c[V_{4i}]=D \cup E$ and $c[V_{4(i-1)+3}]$ is a $k$-subset of $B \cup C$.
\end{itemize}

It is easy to verify that for any index $i$, any two colors $c_1 \in c[V_{4(i-1)+2}]=c[V_{4i}]$ and $c_2 \in c[V_{4(i-1)+3}]$ uniquely determine $s_i$.

We shall show that $c$ is a nonrepetitive coloring of
$P_{\infty}[E_k]$. Assume to the contrary that there is a path
$Q_1Q_2$ in $P_{\infty}[E_k]$ such that the sequence of colors on
$Q_1Q_2$ form a repetition. Remove all vertices from $Q_1Q_2$ that
have colors from the set $X$ and also those vertices which on the
path $Q_1Q_2$ have only neighbors that have colors from the set $X$.
The sequence of colors of the remaining vertices
$Q'_1Q'_2=(q'_{1,1}\ldots q'_{1,l}q'_{2,1}\ldots q'_{2,l})$ still
form a repetition. Let $P'_{\infty}$ be an infinite path with one
loop added to each of its vertices. Furthermore, let $c_S$ be the
coloring of $P'_\infty$ with $c_S(p'_j)=s_j$. Let us define the
function $f:Q'_1Q'_2 \rightarrow P'_\infty$ with $f(q)=p'_i$ if and
only if $q \in v_{4(i-1)+2}[E_k] \cup v_{4(i-1)+3}[E_k] \cup
v_{4i}[E_k]$. Writing $W_1$ and $W_2$ for the images of $Q'_1$ and
$Q'_2$, we obtain that $W_1W_2$ is a walk in $P'_{\infty}$. By the
observation above,  $c_1 \in c[V_{4(i-1)+2}]=c[V_{4i}]$ and $c_2 \in
c[V_{4(i-1)+3}]$ uniquely determine $s_i$.
 This ensures that the color sequence of $W_1W_2$ with respect to $c_S$ is a repetition.
Therefore by \lref{walk} we obtain that $W_1=W_2$.

The remainder of the proof is almost identical to that of \tref{2k1}. Suppose first that $W_1$
and thus $W_2$ contains at least two different vertices. This means that the original paths $Q_1$
and $Q_2$ had to cross from $v_{4(i-1)+2}[E_k] \cup v_{4(i-1)+3}[E_k] \cup v_{4i}[E_k]$ to $v_{4(i)+2}[E_k]
\cup v_{4(i)+3}[E_k] \cup v_{4(i+1)}[E_k]$ or vice versa. But as the layer $v_{4i+1}[E_k]$ is rainbow
colored with colors in $X$, the original color sequence of $Q_1Q_2$ could not be a repetition.

Suppose then that $W_1W_2$ is a walk repeating the same vertex $p'_i$. Then all vertices of $Q_1Q_2$
must lie in $v_{4(i-1)+1}[E_k]\cup v_{4(i-1)+2}[E_k] \cup v_{4(i-1)+3}[E_k] \cup v_{4i}[E_k] \cup v_{4i+1}[E_k]$.
Therefore $Q_1Q_2$ cannot contain any vertex from $\cup v_{4(i-1)+3}[E_k]$ as they have unique colors among vertices
in these 5 layers preventing the possibility of a repetition. By connectivity, we get that $Q_1Q_2$ must lie either
in $v_{4(i-1)+1}[E_k]\cup v_{4(i-1)+2}[E_k]$ or in $v_{4i}[E_k] \cup v_{4i+1}[E_k]$, say the former. By connectivity,
 $Q_1Q_2$ must contain a vertex from $v_{4(i-1)+1}[E_k]$ which has a unique color among vertices
 in $v_{4(i-1)+1}[E_k]\cup v_{4(i-1)+2}[E_k]$. This contradicts the fact that the color sequence of $Q_1Q_2$ is a repetition.

Finally, if the walk $W_1W_2$ is empty, then all vetices of the path $Q_1Q_2$ are either $X$-colored
or all their neighbors in their part of $Q_1Q_2$ are $X$-colored. By connectivity, this is only
possible if all vertices of $Q_1Q_2$ lie with $v_{4i}[E_k] \cup v_{4i+1}[E_k] \cup v_{4i+2}[E_k]$
 for some $i$. Then again by connectivity $Q_1Q_2$ must contain a vertex from $v_{4i+1}[E_k]$.
 This vertex has a unique $c$-color in $v_{4i}[E_k] \cup v_{4i+1}[E_k] \cup v_{4i+2}[E_k]$ thus
  the color sequence of $Q_1Q_2$ with respect to $c$ cannot form a repetition. This contradiction
  completes the proof of \tref{rainbow}.
\end{proof}

\section{Some remarks and open problems}
\label{sec:opprob}

K\"undgen and Pelsmajer \cite{KP} applied their method to
outerplanar graphs. Their techniques can be used to prove the
following theorem.

\begin{theorem}\label{thm:outerpl}
For every outerplanar graph $G$ and integer $k \ge 2$, $\pi(G[K_k]) \le 16 k$. Furthermore,
there exists an outerplanar graph $G_0$ such that $\pi(G_0[E_k]) > 6k$ for every positive integer $k$.
\end{theorem}

\begin{proof} K\"undgen and Pelsmajer \cite{KP} proved that a maximal
outerplanar graph has a shadow complete vertex-partition in which
each $G_k$ is a linear forest. Similarly, we can show that if $G$ is
a maximal outerplanar graph, then $G[K_n]$ has a shadow complete
vertex-partition in which each $G_k$ is of the form $P[K_n]$, where
$P$ is a linear forest. As $\pi(P[K_n]) \le 4k$, it follows from
\tref{4times} that $\pi(G[K_k]) \le 16 k$.

As for the lower bound, in \cite{star1,star2} an outerplanar graph is shown that has star-chromatic
number at least $6$ (a proper vertex coloring is a star-coloring if every path on four vertices uses
at least three distinct colors), thus also nonrepetitive-chromatic number at least $6$.
We can modify this example so that it gives the desired lower bound. Start with a
path $P_{10}$ on $10$ vertices. Add one vertex $u$ connected to all vertices of
$P_{10}$. Then, for each vertex $p_i$ of $P_{10}$ add a 24-vertex path $Q_i$
whose $24$ vertices are all connected to $p_i$. Let us call this the \textit{core}
of our future graph $G_0$. Finally, for every vertex $v$ in the core, let us add $6$ more
leaves $\ell_{v,1},...,\ell_{v,6}$ connected to $v$. Suppose there is a coloring of $G_0[E_k]$
with less than $6k$ colors, we shall arrive to contradiction.

If on the vertices of a layer corresponding to a vertex of the core  there
is a repeated color, then by \lref{degree} we need at least $6k+1$ colors.
Thus we can suppose that the layers corresponding to the vertices of the core
are rainbow colored. The $k$ colors $1,2,\dots k$ used for coloring $u[E_k]$
do not appear on $P_{10}[E_k]$. We call a color \textit{redundant} if it appears
at least on two vertices of $P_{10}[E_k]$. As non-redundant colors are all different,
there are at most $5k$ non-redundant colors. Thus by the pigeon-hole principle there exist
two neighboring layers $p_i[E_k]$ and $p_{i+1}[E_k]$ whose coloring contains at least
one redundant color each. Observe that on $Q_i[E_k]$ the colors $1,2,\dots k$ cannot
appear, as otherwise we would have a repetitive path of length $4$ (through $u[E_k]$
and using the vertices of the redundant color). Also, either on $Q_i[E_k]$ or on
$Q_{i+1}[E_k]$ none of the $2k$ colors of $p_i[E_k]$ and $p_{i+1}[E_k]$ appear,
as otherwise there would be a repetitive path of length $4$ with its endpoints in
$Q_i[E_k]$ and $Q_{i+1}[E_k]$. Suppose that they do not appear on $Q_i[E_k]$.
Thus we can use at most $6k-k-2k=3k$ colors to color $Q_i[E_k]$, but Theorem
\ref{thm:rainbow} implies that we would need at least $3k+1$ colors for this, a contradiction.
\end{proof}

 Tightening the gap between lower and upper bounds in \tref{rainbow},
\tref{complete} and \tref{outerpl}
 are natural open problems related to results in this paper.

Fractional versions of graph parameters have attracted the attention
of researchers. We now introduce a fractional version of
nonrepetitive coloring. For a pair of positive integers $p<q$, a {\em $p$-tuple nonrepetitive $q$-coloring} of $G$
is a mapping $c:V(G) \rightarrow \binom{[q]}{p}$ such that
 for any path $v_1\ldots v_{2l}$ in $G$ the sequence $c_1\ldots c_{2l}$ of colors is not a repetition for any
 choice of $c_i \in c(v_i)$. The {\em fractional Thue chromatic number } $\pi_f(G)$ of a graph $G$ is defined as
\[
\pi_f(G)=\inf\left\{\frac{q}{p}: \exists   \text{a $p$-tuple nonrepetitive $q$-coloring}\ c ~\text{of}\ G\right\}.
\]
By definition, for any graph $G$, $\pi_f(G) \le \pi(G)$. It is easy to see that $\pi_f(P_n)=\pi(P_n)$ for all $n$.
On the other hand, already for the cycle of length $7$, the ordinary Thue chromatic number and the fractional Thue
chromatic number do not coincide as $\pi(C_7)=4$ and $\pi_f(C_7)=3.5$. For the upper bound take the following
$(7,2)$-nonrepetitive coloring of $C_7$:  $v_1\rarrow \{1,2\}; v_2\rarrow \{3,4\}; v_3\rarrow \{1,7\};
v_4\rarrow \{5,6\};v_5\rarrow \{3,4\};v_6\rarrow \{2,6\};v_7\rarrow \{5,7\}$. The lower bound is an elementary case analysis.

\begin{problem}
How big can the   be the difference $\pi(G)-\pi_f(G)$? Is $\pi(G)$
bounded from above by a function of $\pi_f(G)$?
\end{problem}

 For arbitrary graphs, it was proved
\cite{AGHR,dujmovic} that if the maximum degree of $G$ is $\Delta$
then $\pi(G)\le c\Delta^2$ ($c$ is a constant independent of
 $G$ and $\Delta$). This immediately gives that
 $\pi(G[K_k])\le c k^2\Delta^2$, as the maximum degree of $G[K_k]$ is $k(\Delta+1)-1$.
As the graphs $G[K_k]$ have special structure, one may expect that
the upper bound to be improved.   Bar\'at and Wood investigated
nonrepetitive colorings of walks \cite{baratwood}. Following their
definitions, a walk $\{v_1 , v_2 , \ldots , v_{2t}\}$ is \textit{boring} if
$v_i = v_{t+i}$ for all $1\le i\le t$. Clearly, a boring walk is
repetitively colored by every coloring. A coloring $f$ is
\textit{walk-nonrepetitive}
 if only boring walks are repetitively colored by $f$. Let $\pi^W(G)$ denote the least integer such that $G$
 has a walk-nonrepetitive coloring with $\pi^W(G)$ colors. Bar\'at and Wood pose the following problem:
 is there a function $f$ such that $\pi^W(G) \le f(\Delta)$? If this is true, then a rainbow blow-up of
 such a coloring would immediately imply that $\pi_R(G[E_k])\le k\pi^W(G)\le kf(\Delta)$. Indeed a
  repetitive path in $G[E_k]$ would be a 'lift' of a repetitive walk in the original coloring,
  thus boring, which is a contradiction (as the path in $G[E_k]$ cannot be repetitive). It is also
  easy to see that the same coloring would actually show that $\pi(G[K_k])\le k\pi^W(G)\le kf(\Delta)$.

\begin{problem}
Is there a function $f$ such that for every graph $G$ of maximum degree $\Delta$,   $\pi(G[K_k])\le kf(\Delta)$?
Perhaps $\pi(G[K_k])\le c k\Delta^2$ for some constant $c$?
\end{problem}

A natural marriage of the above two notions is the {\em fractional
walk-nonrepetitive chromatic number}, where in the definition of
$p$-tuple nonrepetitive $q$-coloring of $G$, the path $v_1v_2\ldots v_{2l}$
in $G$ is replaced by a walk. We denote by $\pi_f^W(G)$ the
fractional walk-nonrepetitive chromatic number of $G$. It is obvious
that for   path $P$ of length at least $4$,
 $\pi_f^W(P_n) \ge \pi_f(P_n)=\pi(P_n)=3$   and  $\pi_f^W(P_n)\le \pi^W(P_n)\le 4$.
It is also easy to see that $\inf(\pi_R(P_n[E_k])/k)\le
\pi_f^W(P_n)$. A natural question is to determine $\pi_f^W(P_n)$ and also to see
whether equality holds in the previous inequality.

\vskip 0.3truecm

Given a list assignment $L$ with $L(v) \subset \mathbb{N}$ for all
vertices $v$ of a graph $G$, we say that $G$ is $L$-nonrepetitively
colorable if there exists a nonrepetitive coloring $C$ of $G$ with
$c(v) \in L(v)$ for all $v \in V(G)$. The {\em Thue choice number}
$\pi_L(G)$ of a
 graph $G$ is the minimum integer $m$ such that $G$ is $L$-nonrepetitive colorable for every
 list assignment $L$ provided $|L(v)| = m$   for all $v \in V(G)$.
It is known \cite{GPZ2011} that the Thue choice number of a path is
at most $4$. However, the Thue choice number of trees is unbounded
\cite{thuechoose}.

\begin{problem}
Is there a  constant $c$ such that $\pi_L(P_{\infty}[K_k]) \le ck$?
\end{problem}

In the first draft of this paper, we posed the following conjecture,
which has recently been confirmed by Kozik \cite{K}.

\begin{conjecture}
There exists an infinite sequence on four letters, $A,B,C$ and $D$ such that the sequence
is nonrepetitive, palindrome-free and avoids the subsequences $CD$ and $DC$.
\end{conjecture}

\subsubsection*{Acknowledgement} This research was done while the first two authors enjoyed
the hospitality of Zhejiang Normal University.

\end{document}